\documentclass[graybox]{svmult}

\usepackage{mathptmx}       % selects Times Roman as basic font
\usepackage{helvet}         % selects Helvetica as sans-serif font
\usepackage{courier}        % selects Courier as typewriter font
\usepackage{type1cm}        % activate if the above 3 fonts are
\usepackage{makeidx}         % allows index generation
\usepackage{graphicx}        % standard LaTeX graphics tool
\usepackage{multicol}        % used for the two-column index
\usepackage[bottom]{footmisc}% places footnotes at page bottom

\makeindex             % used for the subject index
\usepackage{amssymb,amsmath,amsfonts}

\begin{document}
\title*{Applications of multisymmetric syzygies in invariant theory} 

\author{M. Domokos\thanks{Supported by 
National Research, Development and Innovation Office,  NKFIH K 119934.}}

\institute{MTA Alfr\'ed R\'enyi Institute of Mathematics, 
1053 Budapest, Re\'altanoda utca 13-15, Hungary, 
{\it E-mail address: domokos.matyas@renyi.mta.hu } }
%\date{}

%\subjclass{13A50, 16R30}

\maketitle 

% Make sure theorem have correct destination labels for hyperref.
% Somehow maketitle breaks this if it is done before (labeling only within chapter)
%\renewcommand*{\theHtheorem}{\theHsection.\arabic{theorem}}

\abstract{A presentation by generators and relations of the $n$th symmetric power $B$  of a  commutative algebra $A$ over a field of characteristic zero or greater than $n$ is given. 
This is applied to get information on a minimal homogeneous generating system of $B$ (in the graded case).  The known result that in characteristic zero 
the algebra $B$ is isomorphic to the coordinate ring of the scheme of $n$-dimensional representations of $A$ is also recovered. The special case when $A$ is the two-variable polynomial algebra and $n=3$ is applied to find generators and relations of an algebra of invariants of the symmetric group of degree four that was studied in connection with the problem of classifying sets of four unit vectors in the Euclidean space. } 

{}


\section{Introduction}\label{sec:intro} 

Let $n$ be a positive integer and let $A$ be a commutative $K$-algebra 
(with identity $1$), where $K$ is a field of characteristic zero or $\mathrm{char}(K)>n$. 
Denote by $T^n(A)$ the $n$th tensor power of $A$ on which the symmetric group $S_n$ acts by permuting the tensor factors, and write 
$T^n(A)^{S_n}$ for the subalgebra of $S_n$-invariants. First we give a presentation of this commutative $K$-algebra in terms of generators and relations. 
Note that even if this algebra is finitely generated, we have to take a redundant (typically infinite) generating system that allows a simple description of the relations. 

The key step is to pay attention to  a $K$-vector space basis of $T^n(A)^{S_n}$ 
which comes together with a rule to rewrite products of the generators of $T^n(A)^{S_n}$ into normal form. The rewriting algorithm is  furnished by relations that come uniformly from one master relation. This yields the desired presentation of $T^n(A)^{S_n}$, see Theorem~\ref{thm:firstsecond}. As an application we deduce information on  a minimal generating system of $T^n(A)^{S_n}$ in the case when $A$ is graded, see Corollary~\ref{cor:min_gen_part} and Corollary~\ref{cor:indecomposable}. 
This has relevance for instance for the study of polynomial invariants associated with representations of wreath products, see Section~\ref{sec:wreath}. 

Since the master relation mentioned above comes from the Cayley-Hamilton identity of matrices, as a corollary of Theorem~\ref{thm:firstsecond} we obtain 
Corollary~\ref{cor:gammabar} asserting that 
$T^n(A)^{S_n}$ is isomorphic to the subalgebra   $\mathcal{O}({\mathrm{rep}}(A,n))^{GL(n,K)}$ of $GL(n,K)$-invariants  in the coordinate ring of the scheme of $n$-dimensional representations of $A$ (where $A$ is a finitely generated $K$-algebra and 
$\mathrm{char}(K)=0$). 
This latter result is due to Vaccarino \cite[Theorem 4.1.3]{vaccarino}, who proved it by a different approach. 

In Section~\ref{sec:shapes} we turn to a very concrete application of Theorem~\ref{thm:firstsecond}. Its special case when $A=k[x_1,x_2]$  is a  two-variable polynomial ring was used in  \cite{domokos}, \cite{domokos-puskas} to derive generators of the ideal of relations among a minimal generating system of $T^3(A)^{S_3}$. 
This is applied here to give  a minimal presentation of the ring of invariants 
$R^{S_4}$ of the permutation representation of the symmetric group $S_4$ associated with the action of $S_4$ on the set of two-element subsets of the set $\{1,2,3,4\}$. 
This ring of invariants was studied before by  Aslaksen,  Chan and Gulliksen \cite{aslaksen_etal} because of its relevance for classifying sets of four unit vectors in an euclidean space. A minimal generating system of $R^{S_4}$ was computed in 
\cite{aslaksen_etal}. Here we get simultaneously the generators and relations essentially as a consequence of the minimal presentation of $T^3(K[x_1,x_2])^{S_3}$ mentioned above. 
We note that the ring of invariants $R^{S_4}$ fits into a series that has relevance for the graph isomorphism problem, and has been studied 
for example in \cite{thiery}.

\begin{acknowledgement} I thank Wai-kit Yeung for a question about the topic of Section \ref{sec:semisimple} that inspired me to formulate Theorem~\ref{thm:firstsecond}.
\end{acknowledgement} 

%%%%%%%%%%%%%%%%%%%%%%%%%%%%%%%%%%%%%%%%%%%%%%%%%%%%%%%

\section{Generators and relations for symmetric tensor powers of a commutative algebra}\label{sec:infinite}

Choose a subset $\mathcal{M}\subset A$ such that  $1\notin \mathcal{M}$ and  $\{1\}\cup\mathcal{M}$  is a $K$-vector space basis in $A$. For $w\in \mathcal{M}$ set 
\begin{equation}\label{eq:[w]} 
[w]:=w\otimes 1\otimes\dots\otimes 1+1\otimes w\otimes 1\otimes \dots\otimes 1+\dots+1\otimes\dots\otimes 1\otimes w. 
\end{equation}

\begin{proposition}\label{prop:basis} 
The products $[w_1]\cdots [w_r]$ with $r\leq n$, $w_i\in {\mathcal{M}}$ 
constitute a $K$-vector space basis of $T^n(A)^{S_n}$. 
\end{proposition} 

\begin{proof} 
The elements $w_1\otimes \dots\otimes w_n$ with $w_i\in \{1\}\cup \mathcal{M}$ constitute a basis in $T^n(A)$, and $S_n$ permutes these basis vectors. 
For a multiset $\{w_1,\ldots,w_r\}$ with $r\leq n$, $w_i\in {\mathcal{M}}$, 
denote  by $O_{\{w_1,\ldots,w_r\}}$ the $S_n$-orbit sum of the monomial 
$w_1\otimes \dots\otimes w_r\otimes 1\otimes\dots\otimes 1$, and call $r$ its {\it height}.  Clearly these elements constitute a basis in $T^n(A)^{S_n}$. 
Assume that the multiset $\{w_1,\ldots,w_r\}$ contains $d$ distinct elements with multiplicities $r_1,\ldots,r_d$ (so $r_1+\cdots+r_d=r$), 
then expanding $[w_1]\cdots [w_r]$ as a linear combination of the above basis elements, 
the coefficient of $O_{\{w_1,\ldots,w_r\}}$ is $r_1!\cdots r_d!$, and all other basis elements contributing have strictly smaller height. 
This clearly shows the claim. 
\end{proof} 

\begin{remark} {\rm The special case when $A=K[x_1,\dots,x_m]$ is a polynomial ring and 
$\mathcal{M}$ is the set of monomials appears in \cite[Proposition 2.5']{gelfand-kapranov-zelevinsky}, 
\cite{berele} (see \cite{berele-regev} for the interpretation needed here), and \cite{domokos}. A formally different but related statement is Lemma 3.2 in \cite{vaccarino}.}  
\end{remark}

Take commuting indeterminates $T_w$ ($w\in\mathcal{M}$), and write $\mathcal{F}$ for the commutative polynomial algebra  
$\mathcal{F}:=K[T_w\mid w\in\mathcal{M}]$. 
Write 
\[\varphi:\mathcal{F}\to T^n(A)^{S_n}\] 
for the $K$-algebra homomorphism given by $T_w\mapsto [w]$ for all $w\in\mathcal{M}$. 

To a multiset $\{w_1,\ldots,w_{n+1}\}$ 
of $n+1$ elements from ${\mathcal{M}}$ we associate an element 
$\Psi_{\{w_1\ldots,w_{n+1}\}}\in{\mathcal{F}}$ as follows. 
Write ${\mathcal{P}}_{n+1}$ for the set of partitions 
$\lambda=\lambda_1\bigcup\dots\bigcup \lambda_h$ 
of the set $\{1,\ldots,n+1\}$ into the disjoint union of non-empty subsets 
$\lambda_i$, and denote $h(\lambda)=h$ the number of parts of the partition $\lambda$. 
Set 
$$\Psi_{\{w_1,\ldots,w_{n+1}\}}=\sum_{\lambda\in{\mathcal{P}}_{n+1}} (-1)^{h(\lambda)}
\prod_{i=1}^{h(\lambda)}\left((|\lambda_i|-1)!\cdot 
T_{\prod_{s\in \lambda_i}w_s}\right),$$ 
where for a general element $a\in A$ (say for $a=\prod_{s\in \lambda_i}w_s$) we write 
\[T_a:=c_0n+\sum_{w\in\mathcal{M}} c_wT_w\] 
provided that  $a=c_0+\sum_{w\in\mathcal{M}}c_ww$, with $c_0,c_w\in K$.

\begin{theorem} \label{thm:firstsecond} 
The $K$-algebra homomorphism $\varphi$ is surjective onto $T^n(A)^{S_n}$, and its kernel is the ideal generated by the 
$\Psi_{\{w_1,\ldots,w_{n+1}\}}$, where $\{w_1,\ldots,w_{n+1}\}$ ranges over all multisets of $n+1$ elements in ${\mathcal{M}}$.  
\end{theorem} 

\begin{proof}  Surjectivity of $\varphi$ follows from Proposition~\ref{prop:basis}. The fact 
that $\Psi_{\{w_1,\ldots,w_{n+1}\}}$ belongs to $\ker(\varphi)$ follows from the 
Cayley-Hamilton identity. 
Let $Y(1),\ldots,Y(n+1)$ be  $n\times n$ matrices over an arbitrary commutative ring. 
For a permutation $\pi\in S_{n+1}$ with cycle decomposition  
$$\pi=(i_1\cdots i_d)\cdots (j_1\ldots j_e)$$ 
set 
$${\mathrm{Tr}}^{\pi}={\mathrm{Tr}}(Y(i_1)\cdots Y(i_d))\cdots{\mathrm{Tr}}(Y(j_1)\cdots Y(j_e)).$$ 
Then we have the equality 
\begin{equation}\label{eq:fundtraceid}
\sum_{\pi\in S_{n+1}}{\mathrm{sign}}(\pi){\mathrm{Tr}}^{\pi}=0 
\end{equation} 
called the {\it fundamental trace identity} of $n\times n$ matrices. 
This can be proved by multilinearizing the Cayley-Hamilton identity 
to get an identity multilinear in the $n$ matrix variables $Y(1),\ldots,Y(n)$, and then multiplying by $Y(n+1)$ and taking the trace; see for example \cite{razmyslov}, 
\cite{procesi}, \cite{formanek} for details. 
For $a\in A$ denote by $\tilde{a}$ the diagonal $n\times n$ matrix whose $i$th diagonal entry is $1\otimes\dots\otimes 1\otimes a\otimes 1\otimes\dots\otimes 1$ 
(the $i$th tensor factor is $a$). 
The substitution $Y(i)\mapsto \tilde{w_i}$ ($i=1,\ldots,n+1$) 
in (\ref{eq:fundtraceid}) yields that $\varphi(\Psi_{\{w_1,\ldots,w_{n+1}\}})=0$. 

The coefficient in $\Psi_{\{w_1,\ldots,w_{n+1}\}}$ of the term 
$T_{w_1}\cdots T_{w_{n+1}}$ is $(-1)^{n+1}$, and all other terms are products of 
at most $n$ variables $T_u$. Therefore the relation $\varphi(\Psi_{\{w_1,\ldots,w_{n+1}\}})=0$ can be used to rewrite 
$[w_1]\cdots[w_{n+1}]$ as a linear combination of products of at most $n$ 
invariants of the form $[u]$ (where $u\in{\mathcal{M}}$). 
So these relations are sufficient to rewrite an arbitrary product of the generators $[w]$ in terms of the basis given by Proposition~\ref{prop:basis}. 
This  implies our statement about the kernel of $\varphi$. 
\end{proof} 

\begin{remark}\label{remark:multsym_literature} 
{\rm (i) The ideal of relations among the generators of the algebra of multisymmetric polynomials 
(i.e. the special case $T^n(A)^{S_n}$ with $A=K[x_1,\dots,x_m]$) has been studied by several authors, see \cite{junker1}, \cite{junker2}, \cite{junker3},  \cite{dalbec}, \cite{briand}, \cite{bukhshtaber-rees}, \cite{vaccarino1}.  Slightly more generally, the case of Theorem~\ref{thm:firstsecond} when $A$ is the $q$th Veronese subalgebra of the $m$-variable polynomial algebra $K[x_1,\dots,x_m]$ was given in Theorem 2.5 of \cite{domokos} (whose approach is close to an argument for a different result in \cite{berele}). 
The above proof is a generalization to arbitrary commutative $A$ of the proof of Theorem 2.5 in \cite{domokos}. 

(ii) Although the presentation of $T^n(A)^{S_n}$ given in Theorem~\ref{thm:firstsecond} is infinite, in certain cases an a priori upper bound for the degrees of relations in a minimal presentation is available, and a finite presentation can be obtained from the infinite presentation above (see for instance \cite[Theorem 3.2]{domokos}, building on \cite{derksen}).  
Based on this procedure even a  minimal presentation is worked out in \cite{domokos-puskas} for $T^3(K[x_1,\dots,x_m])^{S_3}$.}
\end{remark}

Suppose that $A$ is generated as a $K$-algebra by the elements $a_1,\dots,a_m$. 
Take the $m$-variable polynomial algebra $K[x_1,\dots,x_m]$ and denote by 
$\rho:K[x_1,\dots,x_n]\to A$ the $K$-algebra surjection given by $\rho(x_i)=a_i$ for $i=1,\dots,m$. This induces a $K$-algebra surjection 
$T^n(\rho):T^n(K[x_1,\dots,x_m])\to T^n(A)$ in the obvious way. 
Since $T^n(\rho)$ is $S_n$-equivariant and $\mathrm{char}(K)$ does not divide $|S_n|$, 
we deduce that it restricts to a $K$-algebra surjections 
$T^n(\rho):T^n(K[x_1,\dots,x_m])^{S_n}\to T^n(A)^{S_n}$. 
Since the algebra $T^n(K[x_1,\dots,x_m])^{S_n}$ is classically known  
to be  generated by 
$[x_{i_1}\dots x_{i_d}]$ where $d\le n$, $1\le i_1\le\dots \le i_d\le m$ 
(see \cite{schlafli},  \cite{macmahon}, \cite{noether}, Chapter II.3 in \cite{weyl}, \cite{stepanov}, 
and for a characteristic free statement see \cite{fleischmann}, \cite{rydh} or  \cite[Corollary 3.14]{vaccarino}),  we obtain the following  known fact: 
\begin{proposition}\label{prop:generators}
The $K$-algebra  $T^n(A)^{S_n}$ is generated by 
\[
\{[a_{i_1}\dots a_{i_d}]\mid \quad d\le n, \quad 1\le i_1\le\dots \le i_d\le m\}.
\]
\end{proposition} 

\begin{remark} 
Proposition~\ref{prop:generators} follows also directly from Theorem~\ref{thm:firstsecond}, 
since the  relation $\varphi(\Psi_{\{w_1,\ldots,w_{n+1}\}})=0$ can be used to rewrite 
$[w_1\cdots w_{n+1}]$ as a linear combination of products of  
invariants of the form $[u]$ where $u$ is a proper subproduct of $w_1\cdots w_{n+1}$. 
\end{remark} 

Moreover, when $A=k[x_1,\dots,x_m]$ is a polynomial ring, the above generating set is minimal. 
For a general commutative $K$-algebra  $A$  the generating set in Proposition~\ref{prop:generators} may not be minimal. As an application of Theorem~\ref{thm:firstsecond} we shall derive some information on a minimal homogeneous  generating system of $T^n(A)^{S_n}$ when $A$ is graded. 
We shall work with graded algebras 
$R=\bigoplus_{d=0}^\infty R_d$, where $R_0=K$. Set $R_+:=\bigoplus_{d>0}R_d$. 
We say that a homogeneous $r\in R_+$ is {\it indecomposable} if $r\notin (R_+)^2$; that is, $r$ is not contained in the subalgebra generated by lower degree elements. 

\begin{corollary} \label{cor:min_gen_part}  Let $A$ be a graded algebra whose minimal positive degree homogeneous component has degree $q$.  
\begin{itemize}
\item[(i)] Suppose that $b$ is a non-zero homogeneous element in $A$ with $\deg(b)<(n+1)q$.  
Then $[b]$ is an indecomposable element in the graded algebra 
$T^n(A)^{S_n}$ (whose grading is induced by the grading on $A$). 
\item[(ii)]\ Suppose that $\mathcal{B}$ is a $K$-linearly independent subset of $A$ consisting of homogeneous elements of positive degree strictly less than $(n+1)q$.  
Then the set $\{[b]\mid b\in \mathcal{B}\}$ is part of a minimal homogeneous 
$K$-algebra generating system of $T^n(A)^{S_n}$. 
\end{itemize}
\end{corollary} 

\begin{proof} Obviously (i) is a special case of (ii), so we shall prove (ii). 
Extend the set $\mathcal{B}$ to a subset $\mathcal{M}\subset A$ such that 
\begin{enumerate} 
\item $\mathcal{M}$ consists of homogeneous elements of positive degree;  
\item $\{1\}\cup \mathcal{M}$ is a $K$-vector space basis of $A$. 
\end{enumerate} 
This is possible by the assumptions. Consider the corresponding (infinite) presentation of $A$ by generators and relations given in Theorem~\ref{thm:firstsecond}.  Endow $\mathcal{F}$ with a grading such that $\deg(T_w)=\deg(w)$ for all $w\in \mathcal{M}$. 
According to Theorem~\ref{thm:firstsecond} the ideal $\ker(\varphi)$ is generated by 
$\Psi_{\{w_1,\ldots,w_{n+1}\}}$, where $\{w_1,\ldots,w_{n+1}\}$ ranges over all multisets of $n+1$ elements in ${\mathcal{M}}$. 
Clearly $\Psi_{\{w_1,\ldots,w_{n+1}\}}\in \mathcal{F}$ is homogeneous of degree 
$\deg(w_1)+\cdots+\deg(w_{n+1})$. For each $w_i$ we have  
$\deg(w_i)\ge q$, implying 
\[\deg(w_1)+\cdots+\deg(w_{n+1})\ge (n+1)q> \deg(b) \mbox{ for all }b\in\mathcal{B}.\]  
Consequently $\ker(\varphi)$ is a homogeneous ideal generated by elements of degree strictly greater than $\deg(b)$ for any $b\in\mathcal{B}$. 
By the graded Nakayama Lemma (see for example \cite[Lemma 3.5.1]{derksen-kemper}) it is sufficient to show that 
the cosets $\{b+(A_+)^2\mid b\in\mathcal{B}\}$ are $K$-linearly independent 
in the factor space $A/(A_+)^2$. This is equivalent to the condition that 
if 
\begin{equation}\label{eq:h}h=\sum_{b\in\mathcal{B}}c_bT_b\in \ker(\varphi)+(\mathcal{F}_+)^2\mbox{ for some }c_b\in K,\end{equation}
then all coefficients $c_b$ are zero, that is, $h=0$. 
The ideals $\ker(\varphi)$ and $(\mathcal{F}_+)^2$ are  
homogeneous, and since all non-zero homogeneous components of $\ker(\varphi)$ have degree strictly grater than $\deg(b)$ for all $b\in\mathcal{B}$, we deduce from \eqref{eq:h} that 
$h\in (\mathcal{F}_+)^2$. Now $h$ is a linear combination of the variables in the polynomial ring $\mathcal{F}$ (in possibly infinitely many variables),  so this leads to the conclusion $h=0$.  
\end{proof}

Proposition~\ref{prop:generators} and Corollary~\ref{cor:min_gen_part} have the following immediate corollary: 

\begin{corollary} \label{cor:indecomposable} 
Suppose that $A$ is a graded algebra generated by homogeneous elements 
of the same positive degree $q$. Let $\mathcal{B}$ be the union of $K$-vector space bases of the positive degree homogeneous components of $A$ of degree strictly less than 
$(n+1)q$ (for example, $\mathcal{B}$ may consist of a set of products of length at most $n$ of the generators of $A$).  Then $\{[b]\mid b\in\mathcal{B}\}$ is a minimal homogeneous $K$-algebra generating system of $T^n(A)^{S_n}$. 
\end{corollary} 

%%%%%%%%%%%%%%%%%%%%%%%%%%%%%

\section{Wreath products}\label{sec:wreath} 

Corollary~\ref{cor:min_gen_part} and Corollary~\ref{cor:indecomposable} can be applied in invariant theory, one of whose basic targets  is to find a minimal homogeneous $K$-algebra generating system in an {\it  algebra of polynomial invariants} $S(V)^G$. Here $G$ is a group acting on a finite dimensional vector space $V$ via linear transformations, and $S(V)$ is the symmetric tensor algebra of $V$ (i.e. a polynomial algebra with a basis of $V$ as the variables) endowed with the induced $G$-action via $K$-algebra automorphisms, and 
$S(V)^G$ is the subalgebra consisting of the elements fixed by $G$. 

For a group $G$ and a positive integer $n$ the {\it wreath product} $G\wr S_n$ 
is defined as the semidirect product $H\rtimes S_n$, where 
$H=G\times \cdots \times G$ is the direct product of $n$ copies of $G$, and conjugation by  $\sigma\in S_n$ yields  the corresponding  permutation of the direct factors of $H$. 
Given a $G$-module $V$ there is a natural $G\wr S_n$-module structure on 
$V^n=V\oplus\cdots\oplus V$ ($n$ direct summands) given by 
\begin{equation}\label{eq:wreath}(g_1,\dots,g_n,\sigma)\cdot(v_1,\dots,v_n)=(g_1 \cdot v_{\sigma^{-1}(1)},\dots,g_n \cdot v_{\sigma^{-1}(n)}).
\end{equation}
Consider the corresponding {\it algebra $S(V^n)^{G\wr S_n}$ of polynomial invariants}. 
Clearly we have 
$S(V^n)^{G\wr S_n}\subseteq S(V^n)^H\subseteq S(V^n)$. 
Since $H$ is normal in $G\wr S_n$, the subalgebra $S(V^n)^H\subseteq S(V^n)$ is 
$S_n$-stable, and 
$S(V^n)^{G\wr S_n}=(S(V^n)^H)^{S_n}$.  
With the usual identification $S(V^n)=T^n(S(V))$ we obtain 
\[S(V^n)^H=S(V^n)^{G\times \cdots\times G}=T^n(S(V)^G),\]
and formula \eqref{eq:wreath} shows that the action on $S(V^n)$ corresponds to the action on $T^n(S(V))$ via permutation of the tensor factors. 
We conclude the identification 
\[S(V^n)^{G\wr S_n}=T^n(S(V)^G)^{S_n}.\] 
Therefore Corollary~\ref{cor:min_gen_part} and Corollary~\ref{cor:indecomposable} have  the following consequence: 

\begin{corollary}\label{cor:wreath}  
Let $q$ denote the minimal positive degree of a homogeneous element in $S(V)^G$, and 
let $\mathcal{B}$ be the union of  $K$-vector space bases of the homogeneous components of  $S(V)^G$ of positive degree strictly less than $(n+1)q$. 
\begin{itemize}
\item[(i)] 
Then $\{[b]\mid b\in\mathcal{B}\}$ is part of a minimal homogeneous $K$-algebra generating system of $S(V^n)^{G\wr S_n}$. 
\item[(ii)]\  Assume in addition that $S(V)^G$ is generated by its homogeneous component of degree $q$. Then $\{[b]\mid b\in\mathcal{B}\}$ is a minimal homogeneous $K$-algebra 
generating system of $S(V^n)^{G\wr S_n}$. 
\end{itemize}
\end{corollary}

\begin{example} (i)  Let $G$ be the special linear group $SL_q(K)$ acting by left multiplication on the space $V=K^{q\times r}$ of $q\times r$ matrices. 
Then by classical invariant theory (see \cite{weyl}) we know that $S(V)^G$ is generated by the determinants of $q\times q$ minors, all having degree $q$, so Corollary~\ref{cor:wreath} (ii) applies for $S(V^n)^{G\wr S_n}$. 

(ii) Let $G$ be the cyclic group of order $q$ acting by multiplication by a primitive $q$th root of $1$ on $V=K^m$. In this case the ring $S(V^n)^{G\wr S_n}$ can be interpreted as the ring of vector invariants of some pseudo-reflection group.
Note that  \cite[Theorem 2.5 (ii)]{domokos} is a special case of Corollary~\ref{cor:wreath} (ii). \end{example} 
%%%%%%%%%%%%%%%%%%%%%%%%%%%%%%%%%%%%%%%%%%%%%%%%%%%%%%%%%%

\section{The scheme of semisimple representations of $A$}\label{sec:semisimple} 

In this section we need to assume that $\mathrm{char}(K)=0$. 
Choose $K$-algebra generators $a_1,\dots,a_m$ of $A$, and consider the $K$-algebra surjection 
$\pi:K\langle x_1,\ldots,x_m\rangle\to A$, $x_i\mapsto a_i$ from the free associative $K$-algebra. 
Take $m$ generic $n\times n$ matrices $X(1),\dots,X(m)$ (their $mn^2$ entries are indeterminates in an $mn^2$-variable polynomial algebra 
$K[x^{(r)}_{ij}\mid 1\le i,j\le n,\quad r=1,\dots,m]$). Take the factor of this polynomial algebra by the ideal generated by all entries of $f(X(1),\dots,X(m))$, where $f$ ranges over $\ker(\pi)$ (in particular, the entries of $X(r)X(s)-X(s)X(r)$ are among the generators of this ideal). This algebra is  the {\it coordinate ring} $\mathcal{O}({\mathrm{rep}}(A,n))$ of the {\it scheme }${\mathrm{rep}}(A,n)$ {\it of $n$-dimensional representations of} $A$ (by definition of this scheme). 
Denote by $Y(1),\dots,Y(m)$ the images in the $n\times n$ matrix algebra over $\mathcal{O}({\mathrm{rep}}(A,n))$ of the generic matrices $X(1),\dots,X(m)$. 
Then the $K$-algebra surjection $\rho:K\langle x_1,\dots,x_m\rangle \to \mathcal{O}({\mathrm{rep}}(A,n))^{n\times n}$ given by $x_i\mapsto Y(i)$ factors through 
a $K$-algebra homomorphism $\bar\rho:A\to \mathcal{O}({\mathrm{rep}}(A,n))^{n\times n}$. 
Recall that the conjugation action of the general linear group $GL(n,K)$ on $n\times n$ matrices induces an action (via $K$-algebra automorphisms) on 
$\mathcal{O}({\mathrm{rep}}(A,n))$. 
Consider the $K$-algebra homomorphism  
\[\gamma:\mathcal{F}\to \mathcal{O}({\mathrm{rep}}(A,n)) \mbox{ given 
by }T_w\mapsto \mathrm{Tr}(\bar\rho(w)).\]  

\begin{corollary}\label{cor:gammabar} 
The $K$-algebra homomorphism $\gamma$ factors through an isomorphism 
$\overline{\gamma}:T^n(A)^{S_n}\to \mathcal{O}({\mathrm{rep}}(A,n))^{GL(n,K)}$
 (so $\gamma=\overline{\gamma}\circ\varphi$). 
 \end{corollary}
 
 \begin{proof} Since the fundamental trace identity holds for matrices over the commutative ring $\mathcal{O}({\mathrm{rep}}(A,n))$, 
 we conclude that all the $\Psi_{\{w_1,\dots,w_{n+1}\}}$ belong to $\ker(\gamma)$. By Theorem~\ref{thm:firstsecond} these elements generate the ideal $\ker(\varphi)$, hence 
 $\ker(\gamma)\supseteq \ker(\varphi)$, implying the existence of a $K$-algebra homomorphism $\overline{\gamma}$ with $\gamma=\overline{\gamma}\circ\varphi$. 
 
 The homomorphism $\gamma$ (and hence $\overline{\gamma}$) is surjective onto $\mathcal{O}({\mathrm{rep}}(A,n))^{GL(n,K)}$ because the algebra of $GL(n,K)$-invariants is generated by 
 traces of monomials in the generic matrices $Y(1),\dots,Y(m)$ (this follows from \cite{sibirskii} since the characteristic of $K$ is zero, hence $GL(n,K)$ is linearly reductive). 

Define $\beta: K[x^{(r)}_{ij}\mid 1\le i,j\le n,\quad r=1,\dots,m]\to T^n(A)$ as the $K$-algebra homomorphism given by $X(r)\mapsto \tilde{a_r}$ (we use the notation of the proof of Theorem~\ref{thm:firstsecond}, so $\tilde{a_r}$ is a diagonal $n\times n$ matrix over $T^n(A)$  whose $j$th diagonal entry is $1\otimes\dots\otimes 1\otimes a_r\otimes 1\otimes\dots\otimes 1$ 
(the $j$th tensor factor is $a_r$). 
Since $f(\tilde{a_1},\dots,\tilde{a_m})=0$ for any $f\in\ker(\pi)$, we conclude that $\beta$ factors through a homomorphism $\overline{\beta}:\mathcal{O}({\mathrm{rep}}(A,n))\to T^n(A)$ inducing $Y(i)\mapsto \tilde{a_i}$. 
It is easy to see that $\overline{\beta}(\overline{\gamma}([w])=\overline{\beta}(\mathrm{Tr}(\bar\rho(w)))=[w]$, hence $\overline{\beta}\circ\overline{\gamma}$ is the identity map 
on $T^n(A)^{S_n}$. This shows that $\overline{\gamma}$ is surjective as well, and hence it is an isomorphism. 
 \end{proof} 
 
The isomorphism $T^n(A)^{S_n}\cong \mathcal{O}({\mathrm{rep}}(A,n))^{GL(n,K)}$ 
is a result of  Vaccarino \cite[Theorem 4.1.3]{vaccarino}. The proof given above is different, and it is an adaptation  of the proof of  \cite[Theorem 4.1]{domokos}.  The special case when $A$ is a polynomial ring is discussed also in 
\cite{gan-ginzburg} and \cite{procesi:2}. 
 Motivated by \cite{artin} we call $\mathcal{O}({\mathrm{rep}}(A,n))^{GL(n,K)}$ the {\it coordinate ring of the scheme of semisimple $n$-dimensional representations of $A$}. 

%%%%%%%%%%%%%%%%%%%%%%%%%%%%

\section{The symmetric group acting on pairs} \label{sec:shapes} 

Write $\binom{[n]}{2}$ for the set of two-element subsets of $\{1,\dots,n\}$. 
The symmetric group $S_n$ acts on the $\binom n2$-variable polynomial algebra 
\[R_n=K[x_{\{i,j\}}\mid \{i,j\}\in \binom{[n]}{2}]\] 
as 
\[\sigma \cdot x_{\{i,j\}}=x_{\{\sigma(i),\sigma(j)\}} \quad \mbox{ for }\sigma\in S_n.\]  
The subalgebra  $R_n^{S_n}$ was studied for two reasons, the first comes from graph theory. 
Given a simple graph $\Gamma$ with vertex set $\{1,\dots,n\}$ and a polynomial $f\in R_n$ 
denote by $f(\Gamma)$ the value of $f$ under the substitution 
\[x_{\{i,j\}}\mapsto \begin{cases} 1 \mbox{ if }\{i,j\}\mbox{ is an edge in }\Gamma \\
0\mbox{ otherwise}.\end{cases}\]
Suppose that $f_1,\dots,f_r$ generate the $K$-algebra $R_n^{S_n}$. 
The following statement is well known (see for examle \cite[Lemma 5.5.1]{derksen-kemper}): 
\begin{proposition}\label{prop:graph-isomorphism}
The graphs $\Gamma$ and $\Gamma'$ on the vertex set $\{1,\dots,n\}$ are isomorphic 
if and only if $f_i(\Gamma)=f_i(\Gamma')$ for all $i=1,\dots,r$. 
\end{proposition} 
 The second motivation to study $R_n^{S_n}$ comes from the problem of classifying sets of $n$ unit vectors in the Euclidean space $\mathbb{R}^m$. We refer to \cite{aslaksen_etal} for the details (see also \cite[Section 5.10.2]{derksen-kemper}).  
 
From now on we focus on the case $n=4$ and set $R:=R_4$. 
For our purposes a more convenient generating system of $R$ is 
$x_1,x_2,x_3,z_1,z_2,z_3$ where 
\[x_1=x_{\{1,2\}}+x_{\{3,4\}},\quad x_2=x_{\{1,3\}}+x_{\{2,4\}},
\quad x_3=x_{\{1,4\}}+x_{\{2,3\}}\]
\[z_1=x_{\{1,2\}}-x_{\{3,4\}},\quad z_2=x_{\{1,3\}}-x_{\{2,4\}},
\quad z_3=x_{\{1,4\}}-x_{\{2,3\}}.\]
We shall use the notation of Section~\ref{sec:infinite}. Take for $A$ the polynomial algebra  $K[x,z]$. Identify $T^3(A)$ with $R$ as follows:  
\[x_1=x\otimes 1\otimes 1,\quad x_2=1\otimes x\otimes 1,\quad x_3=1\otimes 1\otimes x\] 
\[z_1=z\otimes 1 \otimes 1,\quad z_2=1\otimes z\otimes 1,\quad z_3=1\otimes 1\otimes z.\] 
Consequently we have 
\[x_1+x_2+x_3=[x],\quad z_1^2+z_2^2+z_3^2=[z^2],\quad x_1z_1^2+x_2z_2^2+x_3z_3^2=[xz^2], \quad \mbox{etc.}\]

\begin{theorem}~\label{thm:shape_relations} 
\begin{itemize}
\item[(i)] 
The algebra $R^{S_4}$ is  generated by the ten elements 
\[[x],\ [x^2],\ [x^3],\ [z^2],\ [z^4],\ [z^6],\ [xz^2],\ [x^2z^2],\ [xz^4],\ z_1z_2z_3.\] 
\item[(ii)]\  
Consider the surjective $K$-algebra homomorphism 
$\phi$ from the ten-variable polynomial algebra 
\[\mathcal{F}=K[T_w,S \mid w\in\{x,x^2,x^3,y,y^2,y^3,xy,x^2y,xy^2\}]\] 
onto $R^{S_4}$ given by 
\[\phi(S)= z_1z_2z_3\quad  \mbox{ and }\quad  \phi(T_w)= [\psi(w)],\] 
where $\psi:K[x,y]\to K[x,z]$ is the $K$-algebra homomorphism 
mapping $x\mapsto x$ and  $y\mapsto z^2$. 
The kernel of $\phi$ is minimally generated (as an ideal) 
by the element
\begin{equation}\label{eq:S^2}
S^2-\frac 13 T_{y^3}+ \frac 12 T_{y^2}T_{y}-\frac 16 T_{y}^3
\end{equation}  
and the five elements (given explicitly in the proof) 
\begin{equation}\label{eq:J}
J_{3,2},\quad J_{2,3}, \quad J_{4,2}, \quad J_{3,3}, \quad J_{2,4}.
\end{equation}   
\end{itemize}
\end{theorem} 

\begin{proof} Denote by $H$ the abelian normal subgroup of $S_4$ consisting of the three double transpositions and the identity. The variables $x_1,x_2,x_3$ are $H$-invariant whereas $z_1,z_2,z_3$ span $H$-invariant subspaces on which $H$ acts via its three non-trivial characters. It follows that the subalgebra $R^H$ is generated by 
\[x_1,x_2,x_3,z_1^2,z_2^2,z_3^2,z_1z_2z_3.\] 
Denote by $P$ the subalgebra of $R^H$ generated by the six algebraically independent elements $x_1,x_2,x_3,z_1^2,z_2^2,z_3^2$. The square of the remaining generator 
$z_1z_2z_3$ belongs to $P$, hence 
\begin{equation}\label{eq:rk2free} 
R^H=P\oplus Pz_1z_2z_3
\end{equation} 
is a rank two free $P$-module. 
 Since $H$ is normal, $R^H$ is an $S_4$-stable subalgebra of $R$, and the action of $S_4$ on $R^H$ factors through an action of $S_4/H\cong S_3$. More concretely, 
 $S_4$ permutes the elements $x_1,x_2,x_3$ and it permutes the elements $z_1,z_2,z_3$ up to sign. In fact $z_1z_2z_3$ is $S_4$-invariant (as one can easily check) and there exists a surjective group homomorphism $\pi:S_4\to S_3$ (with kernel $H$) such that 
 for any $g\in S_4$ we have 
\[g\cdot x_i=x_{\pi(g)(i)}, \quad g\cdot z_i^2=z_{\pi(g)(i)}^2 \quad (i=1,2,3).\] 
This shows in particular that $P$ is an $S_3=S_4/H$-stable subalgebra of $R^H$, and since 
$z_1z_2z_3$ is  $S_3$-invariant, we deduce from \eqref{eq:rk2free} that 
\begin{equation}\label{eq:rk2invariant} 
R^{S_4}=(R^{H})^{S_3} =P^{S_3}\oplus P^{S_3}z_1z_2z_3
\end{equation} 
is a rank two $P^{S_3}$-module. 
Denote by $\phi'$ the restriction of $\phi$ to the nine-variable polynomial algebra 
\[\mathcal{F}'=K[T_w\mid  w\in\{x,x^2,x^3,y,y^2,y^3,xy,x^2y,xy^2\}].\] 
Then $\phi':\mathcal{F}'\to P^{S_3}$ is a $K$-algebra homomorphism. 
The $K$-algebra homomorphism $\psi:K[x,y]\to K[x,z]$ induces an isomorphism 
\[\tilde\psi:T^3(K[x,y])\to P, \quad x_i\mapsto x_i, \quad y_i\mapsto z_i^2, \quad i=1,2,3\] 
(where we identify $T^3(K[x,y])$ with  $K[x_1,x_2,x_3,y_1,y_2,y_3]$ similarly to the identification of $T^3(K[x,z])$ and $R$). 
Therefore we have 
\begin{equation} \label{eq:fi=mu}\phi'= \tilde\psi \circ \mu
\end{equation} 
where $\mu$ stands for the $K$-algebra surjection 
\[\mu:\mathcal{F}'\to  T^3(K[x,y])^{S_3}, \quad T_w\mapsto [w] \mbox{ for } 
w\in\{x,x^2,x^3,y,y^2,y^3,xy,x^2y,xy^2\}\]
studied in \cite{domokos} and \cite{domokos-puskas}.  In particular, since 
$T^3(K[x,y])^{S_3}$ is known  to be minimally generated by the elements 
$[w]$ with  $w\in \{x,x^2,x^3,y,y^2,y^3,xy,x^2y,xy^2\}$, 
the $K$-algebra homomorphisms $\mu$ and hence $\phi'$ are surjective onto  
$T^3(K[x,y])^{S_3}$, respectively onto $P^{S_3}$. 
By \eqref{eq:rk2invariant} statement (i) of Theorem~\ref{thm:shape_relations} follows. 
Moreover, by \eqref{eq:fi=mu} we have 
\[\ker(\phi')=\ker(\mu).\] 
It was explained first in \cite{domokos} how to deduce from a special case of 
Theorem~\ref{thm:firstsecond} a minimal generating system of $\ker(\mu)$. 
Later in \cite{domokos-puskas}  a natural action of the general linear group $GL_2(K)$ 
was taken into account and it was proved that $\ker(\mu)$ is minimally generated as a 
$GL_2(K)$-ideal by $J_{3,2}$ and $J_{4,2}$ given as follows: 
\begin{align*}
J_{3,2}&=6T_{x^2y}T_{xy}-3T_{xy^2}T_{x^2}-2T_{x^2y}T_{x}T_{y}
+T_{xy^2}T_{x}^2-4T_{xy}^2T_{x}
\\&+2T_{xy}T_{x}^2T_{y}
-3T_{x^3}T_{y^2}+4T_{x^2}T_{x}T_{y^2}-T_{x}^3T_{y^2}
+T_{x^3}T_{y}^2-T_{x^2}T_{x}T_{y}^2
\end{align*} 
\begin{align*}J_{4,2}&=6T_{x^2y}^2+T_{xy}^2T_{x^2}-3T_{xy}^2T_{x}^2-6T_{x^3}T_{xy^2}
+2T_{x^2}T_{xy^2}T_{x}
\\&+4T_{x^3}T_{xy}T_{y}
-2T_{x^2}T_{xy}T_{x}T_{y}+2T_{xy}T_{x}^3T_{y}-4T_{x^2y}T_{x^2}T_{y}
-T_{x^2}^2T_{y^2}
\\&+T_{x^2}^2T_{y}^2+4T_{x^2}T_{x}^2T_{y^2}
-T_{x^2}T_{x}^2T_{y}^2-T_{x}^4T_{y^2}-2T_{x^3}T_{x}T_{y^2}
\end{align*}
The $GL_2(K)$-submodule of $\ker(\mu)$ generated by $J_{3,2}$ is two-dimensional with 
$K$-basis $J_{3,2},J_{2,3}$ where 
\begin{align*}
J_{2,3}&=6T_{xy^2}T_{xy}-3T_{x^2y}T_{y^2}-2T_{xy^2}T_{x}T_{y}
+T_{x^2y}T_{y}^2-4T_{xy}^2T_{y}
\\& +2T_{xy}T_{y}^2T_{x}
-3T_{y^3}T_{x^2}+4T_{y^2}T_{y}T_{x^2}-T_{y}^3T_{x^2}
+T_{y^3}T_{x}^2-T_{y^2}T_{y}T_{x}^2. 
\end{align*} 
The $GL_2(K)$-submodule of $\ker(\mu)$ generated by $J_{4,2}$ is three-dimensional with 
$K$-basis $J_{4,2},J_{3,3},J_{2,4}$ where 
\begin{align*} J_{3,3}
&=3T_{x^2y}T_{xy^2}-T_{xy}T_{x^2}T_{y^2}+T_{xy}^3+T_{xy}T_x^2T_{y^2} 
-5T_{xy}^2T_xT_y-3T_{x^3}T_{y^3}
\\&+2T_{xy}T_{xy^2}T_x+T_{x^2}T_xT_{y^3}
-3T_{x^2}T_{xy^2}T_y+2T_{x^2y}T_{xy}T_y+3T_{x^2}T_xT_{y^2}T_y
\\&+T_{x^3}T_{y^2}T_y
+T_{x^2}T_{xy}T_y^2-T_x^3T_{y^2}T_y+2T_x^2T_{xy}T_y^2-T_{x^2}T_xT_y^3
-3T_xT_{x^2y}T_{y^2} 
\end{align*}
\begin{align*}J_{2,4}&=6T_{xy^2}^2+T_{xy}^2T_{y^2}-3T_{xy}^2T_{y}^2-6T_{y^3}T_{x^2y}
+2T_{y^2}T_{x^2y}T_{y}
\\&+4T_{y^3}T_{xy}T_{x}
-2T_{y^2}T_{xy}T_{x}T_{y}+2T_{xy}T_{y}^3T_{x}-4T_{xy^2}T_{y^2}T_{x}
-T_{y^2}^2T_{x^2}
\\&+T_{y^2}^2T_{x}^2+4T_{y^2}T_{y}^2T_{x^2}
-T_{y^2}T_{y}^2T_{x}^2-T_{y}^4T_{x^2}-2T_{y^3}T_{y}T_{x^2}. 
\end{align*}
Relation \eqref{eq:S^2}  is just Newton's formula expressing the third elementary symmetric polynomial of $z_1^2,z_2^2,z_3^2$ in terms of their power sum symmetric functions. 
It follows from \eqref{eq:rk2invariant} that relation  \eqref{eq:S^2} together with 
$\ker(\phi')$ generate the ideal $\ker(\phi)$. Since \eqref{eq:J} is a minimal generating system of the ideal $\ker(\mu)=\ker(\varphi')$ by \cite{domokos}, \cite{domokos-puskas}, 
the elements \eqref{eq:J} together with \eqref{eq:S^2} constitute a minimal homogeneous generating system of $\ker(\phi)$. 
\end{proof}  

Relation \eqref{eq:S^2} shows that the generator $\phi(T_{y^3})=[z^6]$ of $R^{S_4}$ is redundant. 
Consider the subalgebra 
\[\mathcal{F}_1=K[T_w,S \mid w\in\{x,x^2,x^3,y,y^2,xy,x^2y,xy^2\}]\] 
of $\mathcal{F}$ in Theorem~\ref{thm:shape_relations}. 
A minimal presentation of $R^{S_4}$ in terms of generators and relations is as follows: 

\begin{corollary}\label{cor:mingen}
\begin{itemize}
\item[(i)] 
The algebra $R^{S_4}$ is  minimally generated by the nine elements 
\[[x],\ [x^2],\ [x^3],\ [z^2],\ [z^4],\ [xz^2],\ [x^2z^2],\ [xz^4],\ z_1z_2z_3.\] 
\item[(ii)]\   The kernel of 
the surjective $K$-algebra homomorphism 
\[\phi\vert_{\mathcal{F}_1}: \mathcal{F}_1\to R^{S_4}\] 
is minimally generated (as an ideal) 
by the five elements 
\[\tilde J_{3,2},\quad \tilde J_{2,3}, \quad \tilde J_{4,2},\quad \tilde J_{3,3}, \quad \tilde J_{2,4} \]
obtained from the elements \eqref{eq:J} via the substitution  
\[T_{y^3}\mapsto  3S^2+\frac 32 T_{y^2}T_{y}-\frac 12 T_{y}^3.\] 
\end{itemize}
\end{corollary} 

\begin{remark} {\rm A minimal generating system (as well as the so-called Hironaka decomposition) of $R^{S_4}$ was given by Aslaksen, Chan and Gulliksen \cite[Theorem 4]{aslaksen_etal} by different methods.  The authors also mention that they found the basic syzygies (relations) among the generators with the aid of computer, but the relations  turned out to be quiet complicated, and so they left it out from their paper. } 
\end{remark} 

%%%%%%%%%%%%%%%%%%%%%%%%%%%%%%%%%%%%%%%%%%%

%%%%%%%%%%%%%%%%%%%%%%%%%%%%%%%%%%%%%%%%%%%%%%%%%%%%

\end{document}